\documentclass[onecolumn,10pt]{article}
\usepackage[top=.75in, bottom=.75in, left=.75in, right=.75in]{geometry}
\usepackage{amsmath,amsfonts,amscd,amssymb,amsthm,bbm,bm,booktabs,algorithm,algorithmicx,algpseudocode}
\usepackage{graphicx}
\usepackage{epstopdf}
\usepackage{overpic}
\usepackage{cancel}
\usepackage{rotating}
\usepackage{url}
\usepackage{caption}
\usepackage{color}
\usepackage{multirow}
\usepackage{wrapfig}
\usepackage{mathtools}
\usepackage{subeqnarray}
\usepackage{setspace}
\usepackage{palatino} % needs 10
\usepackage{microtype}
\usepackage[numbers,sort&compress]{natbib}

\usepackage[bottom,flushmargin,hang,multiple]{footmisc}
\usepackage{lipsum}
\newcommand\blfootnote[1]{%
  \begingroup
  \renewcommand\thefootnote{}\footnote{#1}%
  \addtocounter{footnote}{-1}%
  \endgroup
}

\definecolor{header1}{cmyk}{0,0,0,1}

\def \GL {\operatorname{GL}}

\def \R {\mathbbm{R}}
\def \C {\mathbbm{C}}

\numberwithin{equation}{section}
\numberwithin{table}{section}
\numberwithin{equation}{section}
\newtheorem{theorem}{Theorem}[section]

\newtheorem{proposition}[theorem]{Proposition}

\newtheorem{definition}[theorem]{Definition}
\newtheorem{example}[theorem]{Example}
\newtheorem{remark}[theorem]{Remark}

\newcommand{\munarrow}{\mathmakebox[\widthof{1}][c]{\mu}}

\usepackage[utf8]{inputenc}

\usepackage[normalem]{ulem}
\usepackage{color}

\title{\vspace{-.125in}{\huge\selectfont \textbf{On classifications of third- and fourth-order symmetric tensors by eigenvalues}}\vspace{-.075in}}

\author{\normalsize{Lishan Fang and Hua-Lin Huang$^{*}$}\\
\footnotesize{School of Mathematical Sciences, Huaqiao University, Quanzhou 362021, China} 
}

\date{}

\begin{document}
\maketitle

%\subjclass[2010]{15A69, 15A21, 15A18, 13P15}

%\keywords{Eigenvalues; Symmetric tensors; Equivalence class; Binary forms; Canonical forms}

\blfootnote{$^*$ Corresponding author: hualin.huang@hqu.edu.cn}
%%%%%%%%%%%%
%%% ABSTRACT
%%%%%%%%%%%%
\vspace{-.2in}
\begin{abstract}
This paper addresses two fundamental problems posed by Qi regarding the sufficiency of eigenvalues for the classification of symmetric tensors in the two-dimensional setting. For $2\times2\times2$ and $2\times2\times2\times2$ complex symmetric tensors, we classify their equivalence classes using canonical forms of their associated binary cubics and quartics.  We then determine the number of eigenpair classes and the multiplicities of zero eigenvalues for the canonical forms over the complex domain. We demonstrate that these spectral counts do not distinguish all real equivalence classes. Finally, we extend this approach to derive canonical forms and classification for real third- and fourth-order linear partial differential equations in two variables using their bijective relationship to binary forms.
\end{abstract}



\section{Introduction}\label{sec:intro}

The classification of symmetric tensors offers a vital mathematical framework for high-dimensional problems in fields ranging from data analysis to quantum field theory~\cite{landsberg2017geometry}. While eigenvalues have successfully characterized and classified symmetric matrices, extending this paradigm to multilinear algebra presents significant challenges~\cite{horn2012matrix}. In 2006, Qi~\cite{qi2006rank} posed two fundamental and related questions: ``Can we classify third order and fourth order three-dimensional supersymmetric tensors by their eigenvalues?'' and ``Is the classification of cubic and quartic hypersurfaces related with the classification of third and fourth order partial differential equations?'' Although these questions remain open for higher dimensions, this work uses generic eigenpair counts to distinguish the canonical types of two-dimensional complex symmetric tensors and relates their classification to the pointwise classification of principal parts of partial differential equations (PDEs).

The theory of tensor eigenvalues was proposed independently by Qi~\cite{qi2005eigenvalues,qi2006rank,qi2007eigenvalues} and Lim~\cite{lim2005singular} in 2005. Since then, it has found broad applications in theoretical and applied areas, including best rank-one approximation, control, and nonlinear continuum mechanics; see~\cite{li2013characteristic, nie2018real, qi2018tensor} and references therein. Subsequent research has focused on fundamental spectral properties, such as bounds on the number of eigenvalues, the structure of characteristic polynomials~\cite{cartwright2013number, hu2013characteristic, li2013characteristic, ni2007degree}, and the concept of eigenpair class introduced by Cartwright and Sturmfels~\cite{cartwright2013number} to account for the intrinsic scaling freedom. However, recent work demonstrates the insufficiency of an eigenvalue-based approach in the real domain, showing that distinct equivalence classes of third-order symmetric tensors are not uniquely separated by their spectral signatures~\cite{fang2027eigenvalues}. Parallel to these efforts in refining classical boundaries, the scope of tensor spectral theory has expanded to include higher-order structures, such as biquadratic tensors~\cite{qi2025biquadratic} and quaternion tensors~\cite{he2025eigenvalues}. Given that the motivating questions for this study arise within the context of Z-eigenvalues, we adopt the E-eigenvalue framework, the complex extension of Z-eigenvalues, for our classification. Accordingly, throughout this work, eigenvalues and eigenpairs refer specifically to unnormalized E-eigenvalues and E-eigenpairs defined in~\cite{cartwright2013number}.

As a follow-up to~\cite{fang2027eigenvalues}, this paper makes three primary contributions. First, we determine the number of eigenpair classes and the multiplicities of zero eigenvalues for the canonical forms of $2\times2\times2$ and $2\times2\times2\times2$ complex symmetric tensors. This is achieved based on a bijective correspondence between equivalence classes for $2\times2\times2$ and $2\times2\times2\times2$ complex symmetric tensors and the canonical forms of their associated binary cubics and quartics. Second, we demonstrate that this eigenvalue-based approach is insufficient in the real domain, where distinct equivalence classes are shown to share identical spectral signatures. Third, we establish an equivalence between the tensor classification and the classification of third- and fourth-order linear PDEs. Complete canonical forms for such PDEs are provided, thereby addressing Qi's second open question in the two-dimensional case.

The paper proceeds as follows. 
Section~\ref{sec:preliminary} reviews the theory of symmetric tensors and eigenvalues.
Sections~\ref{sec:tensor_third} and~\ref{sec:tensor_fourth} present the equivalence classes and structures of eigenpairs for $2\times2\times2$ and $2\times2\times2\times2$ symmetric tensors, respectively.
Section~\ref{sec:pde} extends the classification results to linear PDEs.
Section~\ref{sec:conclusion} summarizes major findings and future perspectives.

\section{Symmetric tensors and eigenvalues}\label{sec:preliminary}

For the convenience of the reader, we recall the definitions of symmetric tensors and eigenvalues in Subsections~\ref{sec:tensor} and~\ref{sec:eigenvalue}, respectively; more details are given in~\cite{qi2005eigenvalues,qi2018tensor}.

\subsection{Symmetric tensors}\label{sec:tensor}

A complex $m$-th order $n$-dimensional tensor $A$ consists of $n^m$ entries $a_{i_{1} \ldots i_{m}} \in \C$, where $i_{j} = 1, \ldots, n$ for $j=1,\ldots,m$. The tensor $A$ is symmetric if its entries remain invariant under any permutation of indices, which is also known as a ``supersymmetric'' tensor in~\cite{qi2005eigenvalues, qi2006rank}. For clarity, an $m$-th order $n$-dimensional symmetric tensor is referred to as an~$n \times n \times \cdots \times n$ ($m$ times) symmetric tensor. Two tensors $A$ and $B$ are \textit{equivalent}, that is, $A \sim B$, if there exists an invertible matrix~$P \in \GL_n(\C)$ such that
\begin{equation*}
    B=AP^m = \left( \sum_{i_1,\ldots, i_m =1}^{n} a_{i_{1} \ldots i_{m}} p_{i_1j_1}\ldots p_{i_mj_m} \right)_{1 \le j_1,\ldots, j_m \le n}.
\end{equation*}

The symmetric tensor $A$ is associated with a homogeneous polynomial $f(\bm{x})\in\C[x_1,\ldots,x_n]$ of degree $m$, defined as
\[
    f(\bm{x}) \equiv A\bm{x}^m :=\sum_{i_{1}, \ldots, i_{m}=1}^{n} a_{i_{1} \ldots i_{m}}x_{i_1}\ldots x_{i_m},
\]
where $\bm{x}=(x_1,\ldots,x_n)$ is the vector of variables. 
Two polynomials $f$ and $g$ are equivalent if there exists $P\in \GL_n(\C)$ such that 
\begin{equation*}
    g(x)=f(Px) = \sum_{j_1,\ldots, j_m =1}^{n} \sum_{i_1,\ldots, i_m =1}^{n} a_{i_{1} \ldots i_{m}} p_{i_1j_1}\ldots p_{i_mj_m} x_{j_1}\ldots x_{j_m}.
\end{equation*}
Consequently, the equivalence classes of symmetric tensors are in one-to-one correspondence with the canonical forms of their associated homogeneous polynomials.

\subsection{Tensor eigenvalues}\label{sec:eigenvalue}

%We now review some fundamental concepts of tensor eigenvalues essential for our classification framework.

Tensor spectral theory exhibits greater complexity than its matrix counterpart, primarily due to the scaling freedom in the definition. 

%This leads to various types of tensor eigenvalues. The definition of eigenvalues used in this paper is central to the work of Qi~\cite{qi2005eigenvalues,qi2006rank} and others.

\begin{definition}\label{def:eigenvalue}
Let $A=(a_{i_{1} \ldots i_{m}})$ be an $m$-th order $n$-dimensional tensor. If $\lambda \in \C$ and $\bm{x}\in \C^{n}\setminus \{0\}$ satisfy $A\bm{x}^{m-1} = \lambda \bm{x}$, then $\lambda$ is an eigenvalue of $A$ and $\bm{x}$ is an eigenvector of $A$.
\end{definition}

The above definition lacks scale invariance and has motivated early researchers to impose normalization conditions, such as $\bm{x}\cdot \bm{x}=1$ for real eigenvectors and $\bm{x}\cdot \overline{\bm{x}}=1$ for complex eigenvectors~\cite{qi2005eigenvalues, hu2013characteristic, li2013characteristic}. To address the scaling problem, the concept of \textit{eigenpair class} was formalized by Cartwright and Sturmfels~\cite{cartwright2013number}. Two eigenpairs $(\lambda, \bm{x})$ and $(\lambda', \bm{x}')$ are \textit{equivalent} if there exists a nonzero scalar $t\in \C\setminus\{0\}$ such that 
\begin{equation}\label{eqn:eigenpair}
  \begin{cases} 
    \begin{aligned}
      t^{m-2}\lambda=\lambda', \\
      t\bm{x}=\bm{x}'.
    \end{aligned} \\
  \end{cases}
\end{equation}
This equivalence class is essential as eigenvalues are not invariant under $\GL_n(\C)$ transformations unless we restrict to the orthogonal group. Throughout this paper, the number of zero eigenvalues refers to the number of distinct eigenpair classes associated with the eigenvalue zero, rather than its algebraic multiplicity.

This approach works with equivalence classes that encompass all solutions without explicit restriction to normalization, which accounts for the inherent scaling freedom of tensor eigenvalues. It has been proved that the number of such eigenpair classes for a generic $m$-th order $n$-dimensional tensor is $((m-1)^n-1)/(m-2)$, which is also the degree of the associated normalized characteristic equation in $\lambda$ for even $m$ and in $\lambda^2$ for odd $m$. Throughout this paper, we adopt the term ``eigenpair class" to distinguish this equivalence class of eigenpairs from the tensor equivalence classes discussed previously.

%Some researchers~\cite{qi2005eigenvalues, hu2013characteristic, li2013characteristic} impose normalization conditions such as $\bm{x}\cdot \bm{x}=1$ for real eigenvectors and $x\cdot \overline{x}=1$ for complex eigenvectors. Others~\cite{cartwright2013number, li2013characteristic} work with equivalence classes that encompass all solutions without explicit restriction to normalization. These different approaches have made the number of eigenvalues and properties of characteristic polynomials central research questions~\cite{li2013characteristic, ni2007degree, chang2009eigenvalue, qi2005eigenvalues, qi2006rank}.

\begin{remark}\label{rmk:lambda_square}
For odd-order tensors, eigenvalues naturally appear in $\pm \lambda$ pairs due to symmetry. This results in the characteristic equation being a polynomial in $\lambda^2$. However, Equation~\eqref{eqn:eigenpair} implies that $(\lambda,(x,y))$ and $(-\lambda,(-x,-y))$ are collapsed into the same eigenpair class. To maintain the consistency with the definition of the eigenpair class, we continue to use $\lambda$ instead of $\lambda^2$.
\end{remark}

\section{Classifications of $2 \times 2 \times 2$ symmetric tensors}\label{sec:tensor_third}

The classification of $2 \times 2 \times 2$ symmetric tensors follows directly from the classification of binary cubics, where each equivalence class of a tensor corresponds bijectively to a canonical form of its associated homogeneous polynomial. These canonical forms can be derived through classical invariant theory~\cite{gurevich1964foundations, olver1999classical} as well as Harrison's center theory~\cite{fang2027eigenvalues, huang2025solving}. They provide the fundamental tools for our classification framework.
We investigate the canonical forms and classification of these tensors over the complex and real domains in Subsections~\ref{sec:tensor_third_complex} and~\ref{sec:tensor_third_real}, respectively. 

\subsection{$2 \times 2 \times 2$ complex symmetric tensors}\label{sec:tensor_third_complex}

Consider a $2 \times 2 \times 2$ complex symmetric tensor $A=(a_{ijk})_{1\le i,j,k \le 2}$ represented as
\begin{equation*}
A = \left( \begin{pmatrix}
a & b \\
b & c \\
\end{pmatrix} , 
\begin{pmatrix}
b & c \\
c & d \\
\end{pmatrix} \right),  
\end{equation*}
where $a_{111}=a$, $a_{112}=a_{121}=a_{211}=b$, $a_{122}=a_{212}=a_{221}=c$, $a_{222}=d$ and $a,b,c,d \in \C$. The tensor~$A$ corresponds to the binary cubic~$f(x,y) \in \C[x,y]$ given by
\begin{equation*}
    f(x,y) = ax^3+3bx^2y+3cxy^2+dy^3.
\end{equation*}
According to Definition~\ref{def:eigenvalue}, the eigenvalues $\lambda \in \C$ and associated eigenvectors $\bm{x}=(x,y)\in \C^2 \setminus \{0\}$ of $A$ must satisfy the system
\begin{equation}\label{eqn:system_third}
  \begin{cases} 
    \begin{aligned}
      ax^2+2bxy+cy^2 &= \lambda x, \\
      bx^2+2cxy+dy^2 &= \lambda y.
    \end{aligned} \\
  \end{cases}
\end{equation}

We classify these tensors using the canonical forms of their associated binary cubics, which capture all equivalence classes under $\GL_2(\C)$ transformations, as summarized in Table~\ref{tab:binary_cubics}. Since these canonical forms offer simplified polynomial representations, they serve as representatives of their corresponding tensor equivalence classes for later discussions. We now determine the number of eigenpair classes and the multiplicities of zero eigenvalues for each canonical form. The following proposition provides the exact values of these invariants for each canonical form, demonstrating their distinctness.

\begin{table}[h]
    \centering
    \caption{Equivalence classes of $2\times2\times2$ complex symmetric tensors and associated canonical forms}
    \begin{tabular}{ccc}
        \toprule
        Type & Equivalence classes & Canonical forms\\
        \midrule
        1 & $\left( \begin{pmatrix}0 & 0 \\0 & 0 \\\end{pmatrix} ,\begin{pmatrix}0 & 0 \\0 & 0 \\\end{pmatrix} \right)$ & $0$ \\[0.3cm]
        2 & $\left( \begin{pmatrix}1 & 0 \\0 & 0 \\\end{pmatrix} ,\begin{pmatrix}0 & 0 \\0 & 0 \\\end{pmatrix} \right)$ & $x^3$ \\[0.3cm]
        3 & $\left( \begin{pmatrix}1 & 0 \\0 & 0 \\\end{pmatrix} ,\begin{pmatrix}0 & 0 \\0 & 1 \\\end{pmatrix} \right)$ & $x^3+y^3$ \\[0.3cm]
        4 & $\left( \begin{pmatrix}0 & 1 \\1 & 0 \\\end{pmatrix} , \begin{pmatrix}1 & 0 \\0 & 0 \\\end{pmatrix} \right)$ & $3x^2y$ \\
        \bottomrule
    \end{tabular}
    \label{tab:binary_cubics}
\end{table}

\begin{proposition}\label{prop:third_complex}
%the number of eigenpair classes and the multiplicities of zero eigenvalues of $2\times2\times2$ complex symmetric tensors is distinct for each equivalence class, summarized as follows:
Under a generic transformation $P\in\GL_2(\C)$, the number of eigenpair classes and the multiplicities of zero eigenvalues of the canonical forms of $2\times2\times2$ complex symmetric tensors are summarized below:
\begin{center}
        \begin{tabular}{cccc}
        \toprule
        Type & Canonical forms & \# Eigenpair classes & \# Zero eigenvalues \\
        \midrule
        1 & $0$ & infinite & infinite \\
        2 & $x^3$ & 2 & 1 \\
        3 & $x^3+y^3$ & 3 & 0 \\
        4 & $3x^2y$ & 3 & 1 \\
        \bottomrule
    \end{tabular}
\end{center}
\end{proposition}

\begin{proof}

The proof proceeds by calculating the number of eigenpair classes and the multiplicities of zero eigenvalues for each canonical form.

We begin by defining the generic transformation $g(u,v)=f(P(u,v)^T)$ for $P=\begin{pmatrix} p_{1} & p_{2} \\ p_{3} & p_{4} \\\end{pmatrix}\in \GL_{2}(\C)$.
Let $A$ be a $2\times2\times2$ symmetric tensor associated with the binary cubic $g$, whose eigenvalues $\lambda$ and eigenvectors $(u,v)$ satisfy System~\eqref{eqn:system_third}. By eliminating $\lambda$ from the system, we obtain the eigenvector equation
\[
    Q(u,v) = \frac13\left(v \frac{\partial g}{\partial u} - u \frac{\partial g}{\partial v}\right) = 0,
\]
which defines the directions of the eigenvectors. The crux is that if $\lambda=0$, the corresponding eigenvector $(u,v)$ indicates a singularity of $g$, where $\nabla g = 0$.

(1) We first consider~$f(x,y)=0$, whose transformed form remains $g(u,v)=0$. It is evident that $\lambda=0$ is the only solution, and every vector is an eigenvector with $\lambda=0$. Thus, $f$ has infinite eigenpair classes with zero eigenvalues. The zero tensor allows every direction in the complex projective line to serve as an eigenvector, which leads to ``infinite'' eigenpair classes. This also indicates that the classification is exhaustive by including the vanishing cases.

(2) For $f(x,y)=x^3$, the transformed form is $g(u,v) = (p_{1}u+p_{2}v)^3$, whose $\lambda$ and $(u,v)$ must satisfy System~\eqref{eqn:system_third}, written as
\begin{equation*}
  \begin{cases} 
    \begin{aligned}
        p_{1} (p_{1}u+p_{2}v)^2 &= \lambda u, \\
        p_{2} (p_{1}u+p_{2}v)^2 &= \lambda v.
    \end{aligned} \\
  \end{cases}
\end{equation*}
Eliminating $\lambda$ from this system yields an eigenvector equation
\[
    Q(u,v)=(p_1 u+p_2 v)^2(p_1 v-p_2 u)=0.
\]
%This equation has a double root along the line $p_1 u+p_2 v=0$ and a simple root along $p_1 v-p_2 u=0$, which gives two eigenvectors.
For generic $P$, the two lines $p_1u+p_2v=0$ and $p_1v-p_2u=0$ are distinct, giving two eigenvectors.
If $\lambda=0$, it is obvious that $(p_1 u+p_2 v)^2=0$, which corresponds to the line $p_1 u+p_2 v=0$. Thus, this canonical form has two eigenpair classes, and one of them has zero eigenvalues.

(3) Consider the transformed form $g(u,v) = (p_1 u+p_2 v)^3+(p_3 u+ p_4v)^3$.  
%\begin{equation*}\begin{cases} \begin{aligned} p_1(p_1u+p_2 v)^2+p_3(p_3 u+ p_4v)^2-\lambda u &= 0, \\ p_2(p_1u+p_2 v)^2+p_4(p_3 u+ p_4v)^2-\lambda v &= 0 \end{aligned} \\ \end{cases} \end{equation*}
Its eigenvector equation is computed as
\[
    Q(u,v)=(p_1 v-p_2 u)(p_1u+p_2 v)^2+(p_3 v-p_4 u)(p_3 u+ p_4v)^2=0.
\]
Dividing $Q(u,v)$ by $v^3$ and setting $r=u/v$ yields a cubic equation. For generic $P$, this equation has three distinct roots, corresponding to three eigenvectors. It remains to verify whether their associated eigenvalues are zeros. Let $\lambda=0$, the system can be written as
\[
    \begin{pmatrix} p_{1} & p_{3} \\ p_{2} & p_{4} \end{pmatrix} \begin{pmatrix} (p_1u+p_2 v)^2 \\ (p_3 u+ p_4v)^2 \end{pmatrix}= \begin{pmatrix} 0 \\ 0 \end{pmatrix}.
\]
The coefficient matrix of the above system is $P^T$, which is invertible as $P$ is invertible. Thus, the only solution is the trivial solution $p_1u+p_2v=0$ and $p_3u+p_4v=0$, which implies the zero vector $(u,v)=(0,0)$. Therefore, this canonical form has three eigenpair classes with nonzero eigenvalues.

(4) For the transformed form $g(u,v) = 3(p_1 u+p_2 v)^2(p_3 u+ p_4v)$, its eigenvector equation simplifies to
\[
    Q(u,v)=(p_1 u+p_2 v)\left((p_3 v-p_4 u)(p_1 u+p_2 v)+2(p_1 v-p_2 u)(p_3 u+p_4 v)\right)=0,
\]
which has a simple root along the line $p_1 u+p_2 v=0$. If $p_1 u+p_2 v=0$, the gradient of $g$ vanishes, which implies zero eigenvalues. For generic $P$, the second factor gives two distinct roots, corresponding to two new eigenvectors.
Set $\lambda=0$, and the system becomes
\[
    3(p_1u+p_2v) \begin{pmatrix} p_{3} & 2p_{1} \\ p_{4} & 2p_{2} \end{pmatrix} \begin{pmatrix} p_1u+p_2v \\ p_3u+p_4v \end{pmatrix}
    =\begin{pmatrix} 0 \\ 0 \end{pmatrix}.
\]
Since the coefficient matrix is invertible, either $p_1u+p_2v=0$ or $p_1u+p_2v=p_3u+p_4v=0$. The latter gives the invalid zero vector. Thus, the line $p_1u+p_2v=0$ corresponds to the zero eigenvalue, and this canonical form has three eigenpair classes, among which the zero eigenvalue appears in exactly one class.

The analysis confirms the number of eigenpair classes and the multiplicities of zero eigenvalues for the canonical forms under a generic transformation. This completes the proof.
%The analysis confirms the number of eigenpair classes and the multiplicities of zero eigenvalues, which provide a unique spectral signature for the four equivalence classes. This completes the proof of the classification.
\end{proof}

\begin{remark}\label{rmk:fail}
The genericity condition in the proposition cannot be omitted over $\C$. For
\[
P=\begin{pmatrix}1&i\\0&1\end{pmatrix},
\]
the form $x^3$ is transformed into $(u+iv)^3$, whose eigenvector equation has only one eigenpair class, with zero eigenvalue. This particular degeneration has no real analogue, since $p_1^2+p_2^2=0$ has no nonzero real solution. This degeneration is not confined to Type~2, as other types can likewise lose a distinct eigenpair class for a suitably chosen exceptional $P$.
\end{remark}

\begin{remark}\label{rmk:third_singular}
The presence of zero eigenvalues in Proposition~\ref{prop:third_complex} corresponds directly to singularities of the binary cubic $f(x,y)=0$. A zero eigenvalue $\lambda=0$ requires an eigenvector $(x,y)$ where the gradient vanishes, i.e., $\nabla f = 0$. Such a nonzero vector $(x,y)$ determines a singular point of the projective zero set of $f$. The smooth cubic (Type 3) has no singular points and therefore no zero eigenvalues, while the singular cubics (Types 2 and 4) exhibit zero eigenvalues originating from their geometric singularities.
\end{remark}

% counting algebraic multiplicities

For Types 2 to 4, the eigenvector equation is cubic, consistent with the generic count established by Cartwright and Sturmfels~\cite{cartwright2013number}. For Type 2, the eigenvector equation has a repeated root, so the number of distinct eigenpair classes is smaller than the generic count. Apart from the exceptional cases discussed in Remark~\ref{rmk:fail}, this number, combined with the multiplicities of zero eigenvalues, provides a practical criterion for distinguishing the four cases under generic transformations, as shown in Algorithm~\ref{alg:eigenpair}.

\begin{algorithm}[h]
    \caption{Classification of $2\times2\times2$ complex symmetric tensors by eigenvalues}\label{alg:eigenpair}
    \begin{algorithmic}[1]
	\Statex \textbf{Input}: symmetric tensor~$A$
	\Statex \textbf{Output}: canonical type of $A$
	\State Construct the system $A\bm{x}^{m-1}=\lambda \bm{x}$ and solve it to obtain the eigenpair classes $(\lambda, (x,y))$.
        \State Calculate the number of eigenpair classes and the multiplicities of zero eigenvalues.
        \State Identify the canonical type of $A$ based on values listed in Proposition~\ref{prop:third_complex}.
    \end{algorithmic}
\end{algorithm}

We illustrate the classification procedure with two explicit examples.

\begin{example}\label{ex:third_1}
\emph{Let $A = \left( \begin{pmatrix}
1 & 2 \\
2 & 4 \\
\end{pmatrix} , 
\begin{pmatrix}
2 & 4 \\
4 & 8 \\
\end{pmatrix} \right)$, 
which is associated with $f(x,y)=x^3+6x^2y+12xy^2+8y^3$. 
According to Definition~\ref{def:eigenvalue}, its eigenvalues and eigenvectors satisfy 
 \begin{equation*}
  \begin{cases} 
    x^2+4xy+4y^2 = \lambda x, \\
    2x^2+8xy+8y^2 = \lambda y. \\
  \end{cases}
 \end{equation*}
Setting $\lambda=0$ leads to $(x+2y)^2=0$, which implies $x+2y=0$ and an eigenvector takes the form $(-2,1)$.
Assuming $\lambda\ne0$ and eliminating $\lambda$ from the system gives $y-2x=0$. It corresponds to an eigenvector $(1,2)$ and its eigenvalue is calculated as $\lambda=25$. Thus, we have two eigenpair classes
\[
    (0,t(-2,1)), \quad (25t, t(1,2))
\]
for $t\in\C\setminus\{0\}$. Since there are two eigenpair classes with zero eigenvalue occurring in one class, $A$ belongs to Type 2 in Proposition~\ref{prop:third_complex}. By applying the transformation matrix $P = \begin{pmatrix} 1 & -2 \\ 0 & 1 \end{pmatrix}$, the polynomial associated with $A$ is mapped to $f(Pu) = ((u-2v) + 2v)^3 = u^3$, which explicitly verifies that $g(u,v) = u^3$ is the canonical form for this tensor.}
\end{example}

\begin{example}\label{ex:third_3}
\emph{Let $A = \left( \begin{pmatrix}
9 & 6 \\
6 & 6 \\
\end{pmatrix} , 
\begin{pmatrix}
6 & 6 \\
6 & 9 \\
\end{pmatrix} \right)$, 
which is associated with $f(x,y) = 9x^3 + 18x^2y + 18xy^2 + 9y^3$. According to Definition~\ref{def:eigenvalue}, its eigenvalues~$\lambda \in \C$ satisfy 
 \begin{equation*}
  \begin{cases} 
    9x^2 + 12xy + 6y^2 = \lambda x, \\
    6x^2 + 12xy + 9y^2 = \lambda y. \\
  \end{cases}
 \end{equation*}
Solving this system gives three eigenpair classes
\[
    (27t,t(1,1)), \quad \left(\left(\frac{3}{4}+\frac{3\sqrt{7}i}{4}\right)t, t\left(\frac{-3+\sqrt{7}i}{4},1\right)\right), \quad \left(t\left(\frac{3}{4}-\frac{3\sqrt{7}i}{4}\right), t\left(\frac{-3-\sqrt{7}i}{4},1\right)\right).
\]
Since $A$ has three eigenpair classes with nonzero eigenvalues, it belongs to Type 3 according to Proposition~\ref{prop:third_complex}. By applying the transformation matrix $P = \begin{pmatrix} -1/3 & 2/3 \\ 2/3 & -1/3 \end{pmatrix}$, the polynomial associated with $A$ is converted to $f(Pu) = (u+v)(u^2-uv+v^2) = u^3+v^3$, confirming that this tensor belongs to the equivalence class of Type 3.}
\end{example}

\subsection{$2 \times 2 \times 2$ real symmetric tensors}\label{sec:tensor_third_real}

While the eigenvalue-based classification distinguishes the canonical forms over the complex field, it requires refinement in the real domain. The canonical forms of real binary cubics are identical to the complex cases, except for Type 3, which splits into two real orbits $f\sim x^3+y^3$ and $f\sim x^3-3xy^2$. The splitting occurs due to distinct real root configurations, which are equivalent over $\C$, but belong to distinct orbits over $\R$. 

We now demonstrate that distinct real canonical forms can share the same eigenpair structure. This shows that the number of eigenpair classes and the multiplicities of zero eigenvalues do not determine the real equivalence class.

%However, this distinction vanishes over $\C$ due to the algebraic closure.

\begin{example}\label{ex:real_1}
\emph{Let $A_1 = \left( \begin{pmatrix}
1 & 0 \\
0 & 0 \\
\end{pmatrix} , 
\begin{pmatrix}
0 & 0 \\
0 & 1 \\
\end{pmatrix} \right)$, 
which is associated with the binary cubic~$f_1(x,y) = x^3+y^3$. Its eigenvalues and eigenvectors must satisfy the system
\begin{equation*} \begin{cases} 
        x^2 = \lambda x, \\
        y^2 = \lambda y. \\
\end{cases} \end{equation*}
Its eigenvector equation $Q(x,y)=xy(y-x)=0$ gives three real eigenlines, corresponding to three eigenpair classes with nonzero eigenvalues
\[
    (t,t(1,1)), \quad (t,t(0,1)), \quad (t,t(1,0)).
\]}
%According to Proposition~\ref{prop:third_complex}, $A_1$ belongs to Type 3.}
\end{example}

\begin{example}\label{ex:real_2}
\emph{Let $A_2 = \left( \begin{pmatrix}
1 & 0 \\
0 & -1 \\
\end{pmatrix} , 
\begin{pmatrix}
0 & -1 \\
-1 & 0 \\
\end{pmatrix} \right)$, 
which is associated with the form~$f_2(x,y) = x^3-3xy^2$. Its eigenvalues and eigenvectors must satisfy
\begin{equation*} \begin{cases} 
        x^2-y^2 = \lambda x, \\
        -2xy = \lambda y. \\
\end{cases} \end{equation*}
Solving this system yields three eigenpair classes with nonzero eigenvalues, which are
\[
    (t,t(1,0)), \quad \left(t,t\left(-\frac{1}{2},\frac{\sqrt{3}}{2}\right)\right), \quad \left(t,t\left(-\frac{1}{2},-\frac{\sqrt{3}}{2}\right)\right).
\]}
%According to Proposition~\ref{prop:third_complex}, $A_2$ belongs to Type 3.}
\end{example}

\begin{proposition}\label{prop:third_real}
The tensors $A_1$ and $A_2$ belong to distinct $\GL_2(\R)$-equivalence classes but share the same number of eigenpair classes and the multiplicities of zero eigenvalues.
\end{proposition}

\begin{proof}
We first establish that their associated binary cubics $f_1$ and $f_2$ are not equivalent over $\R$. The roots of $f_1(x,y)= x^3+y^3= (x+y)(x^2-xy+y^2)$ consist of one real root and two complex conjugate roots. In contrast, $f_2(x,y) = x^3-3xy^2 = x(x-\sqrt{3}y)(x+\sqrt{3}y)$ factors into three distinct linear factors over $\R$, corresponding to three real roots. Since the number of real roots is invariant under real linear transformation, no $P \in \GL_2(\R)$ can transform $f_1$ into $f_2$. Consequently, their associated tensors $A_1$ and $A_2$ are not equivalent.

However, as calculated in Examples~\ref{ex:real_1} and~\ref{ex:real_2}, both $A_1$ and $A_2$ possess three eigenpair classes, all with nonzero eigenvalues. Thus, these invariant counts fail to distinguish between the two equivalence classes over $\R$.
\end{proof}

\begin{remark}\label{rmk:third_real}
Proposition~\ref{prop:third_real} reveals the insufficiency of this eigenvalue-based approach for classifying real tensors. The number of eigenpair classes and the multiplicities of zero eigenvalues are spectral counts derived from the eigenvector equation. While they distinguish the canonical forms over $\C$, they fail to capture the real root structure of the associated binary cubic form $f$, as shown in~\cite{fang2027eigenvalues}. This deficiency arises because the number of real roots is not determined by these complex spectral counts. Thus, the classification of $2\times2\times2$ real symmetric tensors requires additional $\GL_2(\R)$ invariants to identify the real orbits.
\end{remark}

\section{Classifications of $2 \times 2 \times 2 \times 2$ symmetric tensors}\label{sec:tensor_fourth}

We now extend our eigenvalue-based approach to fourth-order symmetric tensors. The classification of $2\times2\times2\times2$ complex symmetric tensors follows the same fundamental principle using the canonical forms of the associated binary quartics, as shown in Subsection~\ref{sec:tensor_fourth_complex}. We then demonstrate the insufficiency of this approach for real cases in Subsection~\ref{sec:tensor_fourth_real}.

\subsection{$2 \times 2 \times 2 \times 2$ complex symmetric tensors}\label{sec:tensor_fourth_complex}

Consider a $2 \times 2 \times 2 \times 2$ complex symmetric tensor $A=(a_{ijkl})_{1\le i,j,k,l \le 2}$ expressed as
\begin{equation*}
A = \left(\left( \begin{pmatrix}
a & b \\
b & c \\
\end{pmatrix} , 
\begin{pmatrix}
b & c \\
c & d \\
\end{pmatrix} \right),
\left( \begin{pmatrix}
b & c \\
c & d \\
\end{pmatrix} , 
\begin{pmatrix}
c & d \\
d & e \\
\end{pmatrix} \right)\right),
\end{equation*}
where $a_{1111} = a$, $a_{1112} = a_{1121} = a_{1211} = a_{2111} = b$, $a_{1122} = a_{1212} = a_{1221} = a_{2112} = a_{2121} = a_{2211} = c$, $a_{1222} = a_{2122} = a_{2212} = a_{2221} = d$, $a_{2222} = e$ and $a,b,c,d,e \in \C$. The tensor~$A$ corresponds to the binary quartic~$f(x,y) \in \C[x,y]$ defined as
\begin{equation*}
    f(x,y) = ax^4+4bx^3y+6cx^2y^2+4dxy^3+ey^4.
\end{equation*}
According to Definition~\ref{def:eigenvalue}, eigenvalues $\lambda \in \C$ and associated eigenvectors $\bm{x}=(x,y)\in \C^2 \setminus \{0\}$ of $A$ must satisfy the system
\begin{equation}\label{eqn:system_fourth}
  \begin{cases} 
    \begin{aligned}
      ax^3+3bx^2y+3cxy^2+dy^3 &= \lambda x, \\
      bx^3+3cx^2y+3dxy^2+ey^3 &= \lambda y.
    \end{aligned} \\
  \end{cases}
\end{equation}

The canonical forms of such tensors and their associated binary quartics are listed in Table~\ref{tab:binary_quartics}. The values $\mu = \pm \frac{1}{3}$ correspond to degenerate cases that can be further simplified and merged into other types. As with the third-order case, we determine the number of eigenpair classes and the multiplicities of zero eigenvalues for each canonical form. Note that Type 3 is a one-parameter family rather than a single equivalence class.

\begin{table}[h]
    \centering
    \caption{Canonical forms of $2\times2\times2\times2$ complex symmetric tensors and associated binary quartics}
    \begin{tabular}{ccc}
        \toprule
        Type & Tensor representatives & Canonical forms \\
        \midrule
        1 & $\left(\left( \begin{pmatrix}0 & 0 \\0 & 0 \\\end{pmatrix} , \begin{pmatrix}0 & 0 \\0 & 0 \\\end{pmatrix} \right),\left( \begin{pmatrix}0 & 0 \\0 & 0 \\\end{pmatrix} , \begin{pmatrix}0 & 0 \\0 & 0 \\\end{pmatrix} \right)\right)$ & $0$ \\[0.3cm]
        2 & $\left(\left( \begin{pmatrix}1 & 0 \\0 & 0 \\\end{pmatrix} , \begin{pmatrix}0 & 0 \\0 & 0 \\\end{pmatrix} \right),\left( \begin{pmatrix}0 & 0 \\0 & 0 \\\end{pmatrix} , \begin{pmatrix}0 & 0 \\0 & 0 \\\end{pmatrix} \right)\right)$ & $x^4$ \\[0.3cm]
        3 & $\left(\left( \begin{pmatrix}1 & 0 \\0 & \munarrow \\\end{pmatrix} , \begin{pmatrix}0 & \munarrow \\ \munarrow & 0 \\\end{pmatrix} \right),\left( \begin{pmatrix}0 & \munarrow \\ \munarrow & 0 \\\end{pmatrix} , \begin{pmatrix} \munarrow & 0 \\0 & 1 \\\end{pmatrix} \right)\right)$ & $x^4+6\mu x^2y^2+y^4, \mu \ne \pm\frac{1}{3}$ \\[0.3cm]
        4 & $\left(\left( \begin{pmatrix}0 & 0 \\0 & 1 \\\end{pmatrix} , \begin{pmatrix}0 & 1 \\1 & 0 \\\end{pmatrix} \right),\left( \begin{pmatrix}0 & 1 \\1 & 0 \\\end{pmatrix} , \begin{pmatrix}1 & 0 \\0 & 1 \\\end{pmatrix} \right)\right)$ & $6x^2y^2+y^4$ \\[0.3cm]
        5 & $\left(\left( \begin{pmatrix}0 & 0 \\0 & 1 \\\end{pmatrix} , \begin{pmatrix}0 & 1 \\1 & 0 \\\end{pmatrix} \right),\left( \begin{pmatrix}0 & 1 \\1 & 0 \\\end{pmatrix} , \begin{pmatrix}1 & 0 \\0 & 0 \\\end{pmatrix} \right)\right)$ & $6x^2y^2$ \\[0.3cm]
        6 & $\left(\left( \begin{pmatrix}0 & 1 \\1 & 0 \\\end{pmatrix} , \begin{pmatrix}1 & 0 \\ 0 & 0 \\\end{pmatrix} \right),\left( \begin{pmatrix} 1 & 0 \\0 & 0 \\\end{pmatrix} , \begin{pmatrix} 0 & 0 \\0 & 0 \\\end{pmatrix} \right)\right)$ & $4x^3y$\\
         \bottomrule
    \end{tabular}
    \label{tab:binary_quartics}
\end{table}

\begin{proposition}\label{prop:fourth_complex}
Under a generic transformation $P\in\GL_2(\C)$, the number of eigenpair classes and the multiplicities of zero eigenvalues of the canonical forms of $2\times2\times2\times2$ complex symmetric tensors are summarized below:
\begin{center}
    \centering
    \begin{tabular}{cccc}
        \toprule
        Type & Canonical forms & \# Eigenpair classes & \# Zero eigenvalues \\
        \midrule
        1 & $0$ & infinite & infinite \\
        2 & $x^4$ & 2 & 1  \\
        3 & $x^4+6\mu x^2y^2+y^4, \mu \ne \pm \frac{1}{3}$ & 4 & 0  \\
        4 & $6x^2y^2+y^4$ & 4 & 1  \\
        5 & $6x^2y^2$ & 4 & 2   \\
        6 & $4x^3y$ & 3 & 1   \\
         \bottomrule
    \end{tabular}
\end{center}

\end{proposition}

\begin{proof}

Type 1 is trivial, and a proof for Types 2 to 6 is given below. We follow the same approach as in Proposition~\ref{prop:third_complex} to fix a generic $P$, and calculate the number of eigenpair classes and the multiplicities of zero eigenvalues for each canonical form.

(2) For the transformed form $g(u,v) = (p_1 u+p_2 v)^4$, we construct System~\eqref{eqn:system_fourth} and eliminate $\lambda$ to obtain an eigenvector equation as
\[
    Q(u,v)=p_1v(p_1 u+p_2 v)^3-p_2u(p_1 u+p_2 v)^3=(p_1 u+p_2 v)^3(p_1v-p_2u)=0.
\]
For generic $P$, the lines $p_1u+p_2v=0$ and $p_1v-p_2u=0$ are distinct. Thus, this equation has a triple root along the line $p_1 u+p_2 v=0$ and a simple root along $p_1 v-p_2 u=0$, corresponding to two eigenvectors. If $\lambda=0$, the system becomes $(p_1 u+p_2 v)^3=0$. It follows directly that the eigenvector corresponding to $p_1 u+p_2 v=0$ is associated with a zero eigenvalue. Hence, this canonical form has two eigenpair classes, and one class has zero eigenvalues.

(3) For the transformed form $g(u,v) = (p_1 u+p_2 v)^4+6\mu (p_1 u+p_2 v)^2(p_3 u+p_4 v)^2+(p_3 u+p_4 v)^4$, its eigenvector equation $Q(u,v)$ is 
\[
    (p_1v-p_2u)(p_1 u+p_2 v)\left((p_1 u+p_2 v)^2+3\mu(p_3 u+p_4 v)^2\right)+(p_3v-p_4u)(p_3 u+p_4 v)\left((p_3 u+p_4 v)^2+3\mu(p_1 u+p_2 v)^2\right)= 0.
\]
Dividing $Q(u,v)$ by $v^4$ and setting $r=u/v$ yields a quartic equation. For generic $P$, this equation has four distinct roots, corresponding to four eigenvectors.
%By the Fundamental Theorem of Algebra, it has four roots, counting multiplicities, corresponding to four eigenvectors. 

To determine whether their associated eigenvalues are zero, we set $\lambda=0$ and the system becomes $P^T(X,Y)^T=0$, where 
\[
    X=(p_1 u+p_2 v)\left((p_1 u+p_2 v)^2+3\mu(p_3 u+p_4 v)^2\right), \quad Y=(p_3 u+p_4 v)\left((p_3 u+p_4 v)^2+3\mu(p_1 u+p_2 v)^2\right).
\]
%$X=(p_1 u+p_2 v)\left((p_1 u+p_2 v)^2+3\mu(p_3 u+p_4 v)^2\right)$ and $Y=(p_3 u+p_4 v)\left((p_3 u+p_4 v)^2+3\mu(p_1 u+p_2 v)^2\right)$. 
Since $P^T$ is invertible, it follows that $X=0$ and $Y=0$. If
$(p_1u+p_2v)(p_3u+p_4v)\ne0$, then
\[
(p_1u+p_2v)^2=-3\mu(p_3u+p_4v)^2,\qquad
(p_3u+p_4v)^2=-3\mu(p_1u+p_2v)^2,
\]
which implies $9\mu^2=1$, contrary to $\mu\ne\pm\frac{1}{3}$. If
$p_1u+p_2v=0$ or $p_3u+p_4v=0$, then $X=Y=0$ again forces
$(u,v)=(0,0)$. Therefore, this canonical form has four eigenpair classes with nonzero eigenvalues.

(4) Consider the transformed form $g(u,v) = 6(p_1 u+p_2 v)^2(p_3 u+p_4 v)^2+(p_3 u+p_4 v)^4$. Its eigenvector equation is
\[
    Q(u,v)=(p_3 u+p_4 v)\left(3(p_1 v-p_2 u)(p_1 u+p_2 v)(p_3 u+p_4 v)+(p_3 v- p_4 u)\left(3(p_1 u+p_2 v)^2+(p_3 u+p_4 v)^2\right)\right)=0,
\]
which has a simple root along the line $p_3 u+p_4 v=0$ and is associated with $\lambda=0$. For generic $P$, the second factor of $Q(u,v)$ has three distinct roots, corresponding to three new eigenvectors.

It remains to determine the multiplicities of zero eigenvalues. Set
$s=p_1u+p_2v$ and $t=p_3u+p_4v$. Substituting $\lambda=0$ into the system gives
\[
    4t\begin{pmatrix} 3p_{1} & p_{3} \\ 3p_{2} & p_{4} \end{pmatrix}
    \begin{pmatrix} st \\ 3s^2+t^2 \end{pmatrix}
    =\begin{pmatrix} 0 \\ 0 \end{pmatrix}.
\]
Since $P$ is invertible, if $t\ne0$, then $st=0$ and $3s^2+t^2=0$, which has no nonzero solution. Hence $t=0$ gives the only valid eigenvector direction with $\lambda=0$. Therefore, this canonical form has four eigenpair classes, among which the zero eigenvalue appears in exactly one class.

(5) For the transformed form $g(u,v) = 6(p_1 u+p_2 v)^2(p_3 u+p_4 v)^2$, its eigenvector equation is calculated as
\[
    Q(u,v)=3(p_1 u+p_2 v)(p_3 u+p_4 v)\left((p_3 u+p_4 v)(p_1 v-p_2 u)+(p_1 u+p_2 v)(p_3 v-p_4 u)\right)=0,
\]
which has two simple roots along the lines $p_1 u+p_2 v=0$ and $p_3 u+p_4 v=0$. Substituting them into the system gives zero eigenvalues $\lambda=0$. For generic $P$, the third factor gives two distinct roots, corresponding to two eigenvectors.

Substitute $\lambda=0$ into the system, and we have
\[
    12st
    \begin{pmatrix} p_{1} & p_{3} \\ p_{2} & p_{4} \end{pmatrix}
    \begin{pmatrix} t \\ s \end{pmatrix}
    =\begin{pmatrix} 0 \\ 0 \end{pmatrix},
\]
where $s=p_1u+p_2v$ and $t=p_3u+p_4v$. Since $P$ is invertible, $\lambda=0$ occurs precisely when $s=0$ or $t=0$. Thus, the two lines $p_1u+p_2v=0$ and $p_3u+p_4v=0$ correspond to zero eigenvalues, and this canonical form has four eigenpair classes with zero eigenvalue occurring in two classes.

(6) Consider the transformed form $g(u,v) = 4(p_1 u+p_2 v)^3(p_3 u+p_4 v)$. Its eigenvector equation is computed as
\[
    Q(u,v)=(p_1 u+p_2 v)^2\left(3(p_3 u+p_4 v)(p_1v-p_2u)+(p_1 u+p_2 v)(p_3v-p_4u)\right)=0,
\]
which has a double root along $p_1 u+p_2 v=0$, and is associated with a zero eigenvalue. For generic $P$, the second factor of $Q(u,v)$ has two distinct roots, corresponding to two eigenvectors.

Set $\lambda=0$ and express the system as
\[
    4s^2
    \begin{pmatrix} 3p_{1} & p_{3} \\ 3p_{2} & p_{4} \end{pmatrix}
    \begin{pmatrix} t \\ s \end{pmatrix}
    =\begin{pmatrix} 0 \\ 0 \end{pmatrix},
\]
where $s=p_1u+p_2v$ and $t=p_3u+p_4v$. Since $P$ is invertible, either $s=0$ or the system has only the trivial solution. Hence, the line $p_1u+p_2v=0$ corresponds to the zero eigenvalue, while the two eigenvectors from the second factor have nonzero eigenvalues. Therefore, this canonical form has three eigenpair classes, and only one class has zero eigenvalue.

%Based on the above analysis, we observed that each equivalence class associated with the canonical forms possesses a distinct eigenpair structure as summarized in the table, which yields a unique criterion for the classification. This completes the proof.

\end{proof}

\begin{remark}
Type~3 is a family rather than a single equivalence class. For example, the invariant $\eta(f)=I^3/(I^3-27J^2)$, where $I=ae-4bd+3c^2$ and $J=ace+2bcd-ad^2-b^2e-c^3$, takes different values on $x^4+y^4$ and $x^4+12x^2y^2+y^4$, showing that these two forms are inequivalent, although both have four eigenpair classes and no zero eigenvalue.
\end{remark}

\begin{remark}\label{rmk:fourth_isotropic}
As in Remark~\ref{rmk:fail}, the genericity condition in
Proposition~\ref{prop:fourth_complex} cannot be omitted: for
$P=\begin{pmatrix}1&i\\0&1\end{pmatrix}$, the form $x^4$ becomes
$(u+iv)^4$, which has one eigenpair class instead of two. The same applies to other types, including Type 3.
\end{remark}

%\begin{remark}\label{rmk:fourth_singular}
%The correspondence between zero eigenvalues and singularities of $f(x,y)=0$ extends to the quartic case. As shown in Proposition~\ref{prop:fourth_complex}, the smooth binary quartic is non-singular and has no zero eigenvalues, while all other types are singular. Their distinct singularity structures—from the single cusp of Type 2 to the two nodes of Type 5 dictate the number of eigenpair classes with a zero eigenvalue, establishing a one-to-one correspondence between the algebraic eigenpair structure and the geometric singularities of the curve.
%\end{remark}

The following examples illustrate the classification of $2 \times 2 \times 2 \times 2$ complex symmetric tensors using Proposition~\ref{prop:fourth_complex}. They follow the same eigenvalue-based procedure developed for third-order tensors.

%The procedure for classifying $2 \times 2 \times 2 \times 2$ symmetric tensors up to equivalence is the same as the third-order case. Two explicit examples for computing eigenpair classes and equivalence classes for such tensors are given below for illustration.

\begin{example}\label{ex:fourth_1}
\emph{Let $A = \left(\left( \begin{pmatrix}
16 & 8 \\
8 & 4 \\
\end{pmatrix} , 
\begin{pmatrix}
8 & 4 \\
4 & 2 \\
\end{pmatrix} \right),\left( \begin{pmatrix}
8 & 4 \\
4 & 2 \\
\end{pmatrix} , 
\begin{pmatrix}
4 & 2 \\
2 & 1 \\
\end{pmatrix} \right)\right)$, 
which corresponds to $f(x,y) = 16x^4 + 32x^3y + 24x^2y^2 + 8xy^3 + y^4$. According to Definition~\ref{def:eigenvalue}, its eigenvalues and eigenvectors satisfy 
 \begin{equation*}
  \begin{cases} 
    16x^3 + 24x^2y + 12xy^2 + 2y^3 = \lambda x, \\
    8x^3 + 12x^2y + 6xy^2 + y^3 = \lambda y. \\
  \end{cases}
 \end{equation*}
Solving this system gives two eigenpair classes 
\[
    (0, t(-1,2)), \quad (125t^2,t(2,1)).
\]
Since $A$ has two eigenpair classes with one zero eigenvalue, it belongs to Type 2 in Proposition~\ref{prop:fourth_complex}.}
\end{example}

\begin{example}\label{ex:fourth_2}
\emph{Let $A = \left(\left( \begin{pmatrix}
2 & 1 \\
1 & 1 \\
\end{pmatrix} , 
\begin{pmatrix}
1 & 1 \\
1 & 1 \\
\end{pmatrix} \right),\left( \begin{pmatrix}
1 & 1 \\
1 & 1 \\
\end{pmatrix} , 
\begin{pmatrix}
1 & 1 \\
1 & 2 \\
\end{pmatrix} \right)\right)$, 
which corresponds to $f(x,y) = 2x^4 + 4x^3y + 6x^2y^2 + 4xy^3 + 2y^4$.
According to Definition~\ref{def:eigenvalue}, its eigenvalues and eigenvectors satisfy 
 \begin{equation*}
  \begin{cases} 
    2x^3 + 3x^2y + 3xy^2 + y^3 = \lambda x, \\
    x^3 + 3x^2y + 3xy^2 + 2y^3 = \lambda y. \\
  \end{cases}
\end{equation*}
We eliminate $\lambda$ from the system and obtain four eigenpair classes
\[
    (9t^2,t(1,1)), \quad (t^2, t(-1,1)), \quad \left(0, t\left(\frac{-1+\sqrt{3}i}{2},1\right)\right), \quad \left(0, t\left(\frac{-1-\sqrt{3}i}{2},1\right)\right).
\]
Since $A$ has four eigenpair classes and two of them contain zero eigenvalues, it belongs to Type 5 in Proposition~\ref{prop:fourth_complex}.}
\end{example}

\subsection{$2 \times 2 \times 2 \times 2$ real symmetric tensors}\label{sec:tensor_fourth_real}

Following the classification of $2 \times 2 \times 2 \times 2$ symmetric tensors over $\C$, we now consider such tensors over $\R$. The classification of real binary quartics is finer than the complex case, and the complete classification contains sixteen canonical forms~\cite{gurevich1964foundations, olver1999classical}. The complex Type 3 family gives rise to four one-parameter real families, represented by Types 3 to 6 in Table~\ref{tab:pde_fourth}. Setting $\mu=0$ in $x^4+y^4$ and $x^4-y^4$ gives two real forms with distinct root structures but identical counts of complex eigenpair classes. Thus, these spectral counts do not determine the real equivalence class.

\begin{example}\label{ex:fourth_real_1}
\emph{Let $A_3 = \left(\left( \begin{pmatrix}
1 & 0 \\
0 & 0 \\
\end{pmatrix} , 
\begin{pmatrix}
0 & 0 \\
0 & 0 \\
\end{pmatrix} \right),\left( \begin{pmatrix}
0 & 0 \\
0 & 0 \\
\end{pmatrix} , 
\begin{pmatrix}
0 & 0 \\
0 & 1 \\
\end{pmatrix} \right)\right)$, 
which is associated with $f_3(x,y) = x^4+y^4$. Its eigenvalues and eigenvectors must satisfy
\begin{equation*} \begin{cases} 
        x^3 = \lambda x, \\
        y^3 = \lambda y. \\
\end{cases} \end{equation*}
The resulting eigenvector equation $Q(x,y) = x^3y-xy^3 = xy(x+y)(x-y) = 0$ gives four eigenpair classes with nonzero eigenvalues
\[
    (t^2,t(1,0)), \quad (t^2,t(0,1)), \quad (t^2,t(1,1)), \quad (t^2,t(1,-1)).
\]}
%According to Proposition~\ref{prop:third_complex}, $A_1$ belongs to Type 3.}
\end{example}

\begin{example}\label{ex:fourth_real_2}
\emph{Let $A_4 = \left(\left( \begin{pmatrix}
1 & 0 \\
0 & 0 \\
\end{pmatrix} , 
\begin{pmatrix}
0 & 0 \\
0 & 0 \\
\end{pmatrix} \right),\left( \begin{pmatrix}
0 & 0 \\
0 & 0 \\
\end{pmatrix} , 
\begin{pmatrix}
0 & 0 \\
0 & -1 \\
\end{pmatrix} \right)\right)$, 
which is associated with $f_4(x,y) = x^4-y^4$. Its eigenvalues and eigenvectors must satisfy
\begin{equation*} \begin{cases} 
        x^3 = \lambda x, \\
        -y^3 = \lambda y. \\
\end{cases} \end{equation*}
Its eigenvector equation $Q(x,y) = xy(x^2+y^2) = 0$ gives four eigenpair classes with nonzero eigenvalues as
\[
    (t^2,t(1,0)), \quad (t^2,t(0,i)), \quad (t^2,t(1,i)), \quad (t^2,t(1,-i)).
\]}
%According to Proposition~\ref{prop:third_complex}, $A_2$ belongs to Type 3.}
\end{example}

\begin{proposition}\label{prop:fourth_real}
The tensors $A_3$ and $A_4$ belong to distinct $\GL_2(\R)$-equivalence classes but have the same number of eigenpair classes and the multiplicities of zero eigenvalues.
\end{proposition}

\begin{proof}
We first prove that the binary quartics $f_3$ and $f_4$ are not equivalent over $\R$. The roots of $f_3(x,y)= x^4+y^4=(x-e^{i\pi/4} y)(x-e^{i3\pi/4} y)(x-e^{i5\pi/4} y)(x-e^{i7\pi/4} y)$ are four complex roots. In contrast, $f_4(x,y) = x^4-y^4=(x-y)(x+y)(x-iy)(x+iy)$ can be factorized into two distinct real linear factors and two complex conjugate factors, corresponding to two distinct real roots and two complex conjugate roots. It follows that no $P \in \GL_2(\R)$ can transform $f_3$ into $f_4$. Therefore, their associated tensors $A_3$ and $A_4$ are not equivalent over $\R$.

However, the calculation in Examples~\ref{ex:fourth_real_1} and~\ref{ex:fourth_real_2} demonstrates that both $A_3$ and $A_4$ have four eigenpair classes, all with nonzero eigenvalues. In other words, these spectral counts cannot uniquely determine $\GL_2(\R)$-equivalence classes.
\end{proof}

\begin{remark}
The above results follow the principle for the third-order case in Remark~\ref{rmk:third_real}. The number of eigenpair classes and the multiplicities of zero eigenvalues are insufficient to separate real orbits of fourth-order tensors, as shown in Proposition~\ref{prop:fourth_real}. Therefore, classifying $2\times2\times2\times2$ real symmetric tensors also requires additional $\GL_2(\R)$ invariants.
\end{remark}

\section{Classifications of third- and fourth-order PDEs}\label{sec:pde}

In this section, we derive the canonical forms of third- and fourth-order linear PDEs based on the equivalence between the classification of binary forms and that of PDEs. The classification of second-order PDEs is well-established and has been extended to third order to address more challenging problems~\cite{dzhuraev1991classification, evans2022partial}. Let $L_m[U] = \Phi$ be an $m$-th order linear PDE, where $U$ is a function, $L_m[U]$ represents the principal part (terms with the highest-order derivatives), and $\Phi$ represents all lower-order terms. The classification is primarily determined by the principal part $L_m[U]$, but some lower-order terms may be included for finer classification~\cite{ben2003dispersion}. The following analysis focuses on canonical forms based on $L_m[U]$.

The classification results for real third- and fourth-order linear PDEs, derived via the correspondence to binary forms, are presented in Tables~\ref{tab:pde_third} and~\ref{tab:pde_fourth}, respectively. In these tables, the magnitudes of the nonzero leading coefficients of the canonical PDE representatives are normalized to 1, while their signs are retained to distinguish the real forms. Note that for real PDEs, the classification requires distinguishing between real orbits, which may merge under complex equivalence, as discussed in Section~\ref{sec:tensor_third_real}.

\begin{table}[h]
    \centering
    \caption{Canonical forms of real binary cubics and third-order linear PDEs}
    \begin{tabular}{ccc}
        \toprule
        Type & Canonical forms (binary cubics) & Canonical forms (PDEs) \\
        \midrule      
        1 & $0$ & $0 = \Phi$ \\
        2 & $x^3$ & $\partial_{x}^3U = \Phi$ \\
        3 & $x^3+y^3$ & $\partial_{x}^3U+\partial_{y}^3U = \Phi$ \\
        4 & $x^3-3xy^2$ & $\partial_{x}^3U-3\partial_{x}\partial_{y}^2U = \Phi$ \\
        5 & $ 3x^2y$ & $3\partial_{x}^2\partial_{y}U = \Phi$ \\
        \bottomrule
    \end{tabular}
    \label{tab:pde_third}
\end{table}

\begin{table}[h]
    \centering
    \caption{Canonical forms of real binary quartics and fourth-order linear PDEs}
    \begin{tabular}{ccc}
        \toprule
        Type & Canonical forms (binary quartics) & Canonical forms (PDEs) \\
        \midrule        
        1 & $0$ & $0 = \Phi$ \\
        2 & $x^4$ & $\partial_{x}^4U = \Phi$ \\
        3 & $x^4+6\mu x^2y^2+y^4,\ -\frac{1}{3}<\mu<\frac{1}{3}$ 
          & $\partial_{x}^4U+6\mu \partial_{x}^2\partial_{y}^2U+\partial_{y}^4U = \Phi,\ -\frac{1}{3}<\mu<\frac{1}{3}$ \\
        4 & $-\left(x^4+6\mu x^2y^2+y^4\right),\ -\frac{1}{3}<\mu<\frac{1}{3}$ 
          & $-\partial_{x}^4U-6\mu \partial_{x}^2\partial_{y}^2U-\partial_{y}^4U = \Phi,\ -\frac{1}{3}<\mu<\frac{1}{3}$ \\
        5 & $x^4+6\mu x^2y^2-y^4,\ \mu\in\R$
          & $\partial_{x}^4U+6\mu \partial_{x}^2\partial_{y}^2U-\partial_{y}^4U = \Phi,\ \mu\in\R$ \\
        6 & $x^4+6\mu x^2y^2+y^4,\ \mu<-\frac{1}{3}$ 
          & $\partial_{x}^4U+6\mu \partial_{x}^2\partial_{y}^2U+\partial_{y}^4U = \Phi,\ \mu<-\frac{1}{3}$ \\
        7 & $x^2y^2+y^4$ 
          & $\partial_{x}^2\partial_{y}^2U+\partial_{y}^4U = \Phi$ \\
        8 & $-x^2y^2-y^4$ 
          & $-\partial_{x}^2\partial_{y}^2U-\partial_{y}^4U = \Phi$ \\
        9 & $-x^2y^2+y^4$ 
          & $-\partial_{x}^2\partial_{y}^2U+\partial_{y}^4U = \Phi$ \\
        10 & $x^2y^2-y^4$ 
           & $\partial_{x}^2\partial_{y}^2U-\partial_{y}^4U = \Phi$ \\
        11 & $(x^2+y^2)^2$ 
           & $(\partial_{x}^2+\partial_{y}^2)^2U = \Phi$ \\
        12 & $-(x^2+y^2)^2$ 
           & $-(\partial_{x}^2+\partial_{y}^2)^2U = \Phi$ \\
        13 & $x^2y^2$ 
           & $\partial_{x}^2\partial_{y}^2U = \Phi$ \\
        14 & $-x^2y^2$ 
           & $-\partial_{x}^2\partial_{y}^2U = \Phi$ \\
        15 & $x^3y$ 
           & $\partial_{x}^3\partial_{y}U = \Phi$ \\
        16 & $-x^4$ 
           & $-\partial_{x}^4U = \Phi$ \\
        \bottomrule
    \end{tabular}
    \label{tab:pde_fourth}
\end{table}

To establish these canonical forms, we explicitly define the correspondence between the principal parts and their associated binary forms. Consider a third-order linear PDE in two variables, whose principal part is represented by
\begin{equation*}
    L_3[U]=a(x,y)\partial_{x}^3U+3b(x,y)\partial_{x}^2\partial_{y}U+3c(x,y)\partial_{x}\partial_{y}^2U+d(x,y)\partial_{y}^3U.
\end{equation*}
At a fixed point, the principal part corresponds to the binary cubic principal symbol
\[
f_3(\xi,\eta)=a\xi^3+3b\xi^2\eta+3c\xi\eta^2+d\eta^3,
\]
where $a,b,c,d$ denote the coefficients evaluated at that point. Similarly, the principal part of a fourth-order linear PDE in two variables is defined as
\begin{equation*}
    L_4[U]=a(x,y)\partial_{x}^4U+4b(x,y)\partial_{x}^3\partial_{y}U+6c(x,y)\partial_{x}^2\partial_{y}^2U+4d(x,y)\partial_{x}\partial_{y}^3U+e(x,y)\partial_{y}^4U.
\end{equation*}
At a fixed point, the principal part corresponds to the binary quartic principal symbol
\[
f_4(\xi,\eta)=a\xi^4+4b\xi^3\eta+6c\xi^2\eta^2+4d\xi\eta^3+e\eta^4,
\]
where $a,b,c,d,e$ denote the coefficients evaluated at that point. 
%Note that for local classification at a specific point, the coefficient functions $a, b, c, d, e$ are treated as constants.

% and in turn, its associated $2\times2\times2$ symmetric tensor
% and its associated $2\times2\times2\times2$ symmetric tensor

Dzhuraev and Popelek~\cite{dzhuraev1991classification} established a classification of third-order PDEs based on the discriminant of the characteristic equation. It was extended to include quadratic terms for a richer set of canonical forms in~\cite{ben2003dispersion}. Based on the one-to-one correspondence between linear PDEs and binary forms, the canonical forms of PDEs can be derived directly based on those of binary forms. We first establish the equivalence between the classification of third- and fourth-order linear PDEs in two variables and the classification of binary cubics and quartics in the following propositions.

\begin{proposition}\label{prop:pde_third}
The classification of third-order linear PDEs in two variables by their principal parts under an invertible linear change of coordinates is equivalent to the classification of binary cubics under $\GL_{2}(\R)$ transformation.
\end{proposition}

\begin{proof}
Consider the principal part of the third-order linear PDE $L_3[U]$ and the binary cubic $f_3(x,y)$ as defined above. Let
$P=\begin{pmatrix}p_1&p_2\\p_3&p_4\end{pmatrix}\in\GL_2(\R)$ and parameterize the corresponding linear change of coordinates by
$(x,y)^T=P^{-T}(u,v)^T$. By the chain rule,
\[
    \partial_x=p_1\partial_u+p_2\partial_v,\qquad
    \partial_y=p_3\partial_u+p_4\partial_v.
\]
Substituting these operators into $L_3[U]$ and expanding gives the transformed principal part
\[
    L_{3}'[U]=A(u,v)\partial_u^3U+3B(u,v)\partial_u^2\partial_vU
    +3C(u,v)\partial_u\partial_v^2U+D(u,v)\partial_v^3U,
\]
where
\begin{equation*}
    \begin{aligned}
      A(u,v) &= a p_1^3 + 3b p_1^2 p_3 + 3c p_1 p_3^2 + d p_3^3, \\
      B(u,v) &= a p_1^2 p_2 + b(p_1^2 p_4 + 2p_1 p_2 p_3) + c(2p_1 p_3 p_4 + p_2 p_3^2) + d p_3^2 p_4, \\
      C(u,v) &= a p_1 p_2^2 + b(2p_1 p_2 p_4 + p_2^2 p_3) + c(p_1 p_4^2 + 2p_2 p_3 p_4) + d p_3 p_4^2, \\
      D(u,v) &= a p_2^3 + 3b p_2^2 p_4 + 3c p_2 p_4^2 + d p_4^3.
    \end{aligned} \\
\end{equation*}
Since the transformation is linear, no lower-order terms are generated. At the fixed point, $a,b,c,d$ are constants. Under the induced action on the principal symbol,
\[
f'_3(\xi,\eta)=f_3(P(\xi,\eta)^T)
=A\xi^3+3B\xi^2\eta+3C\xi\eta^2+D\eta^3.
\]
Thus, the association between $L_3$ and $f_3$ preserves the $\GL_2(\R)$ action and establishes a one-to-one correspondence between their canonical forms.
\end{proof}

\begin{proposition}\label{prop:pde_fourth}
The classification of fourth-order linear PDEs in two variables by their principal parts under an invertible linear change of coordinates is equivalent to the classification of binary quartics under $\GL_{2}(\R)$ transformation.
\end{proposition}

\begin{proof}

Consider the principal part of the fourth-order linear PDE $L_4[U]$ and the binary quartic $f_4$ defined above. Let
$P=\begin{pmatrix}p_1&p_2\\p_3&p_4\end{pmatrix}\in\GL_2(\R)$ and parameterize the linear change of coordinates by
$(x,y)^T=P^{-T}(u,v)^T$. By the chain rule,
\[
    \partial_x=p_1\partial_u+p_2\partial_v,\qquad
    \partial_y=p_3\partial_u+p_4\partial_v.
\]
Substituting these operators into $L_4[U]$ and expanding gives the transformed principal part
\[
    L_{4}'[U] = A(u,v)\partial_{u}^4U + 4B(u,v)\partial_{u}^3\partial_{v}U + 6C(u,v)\partial_{u}^2\partial_{v}^2U + 4D(u,v)\partial_{u}\partial_{v}^3U + E(u,v)\partial_{v}^4U,
\]
where
\begin{equation*}
    \begin{aligned}
    A(u,v) &= a p_1^4 + 4b p_1^3 p_3 + 6c p_1^2 p_3^2 + 4d p_1 p_3^3 + e p_3^4, \\
    B(u,v) &= a p_1^3 p_2 + b(3 p_1^2 p_2 p_3 + p_1^3 p_4) + c(3 p_1 p_2 p_3^2 + 3 p_1^2 p_3 p_4) + d(p_2 p_3^3 + 3 p_1 p_3^2 p_4) + e p_3^3 p_4, \\
    C(u,v) &= a p_1^2 p_2^2 + 2b (p_1^2 p_2 p_4 + p_1 p_2^2 p_3) + c (p_1^2 p_4^2 + 4 p_1 p_2 p_3 p_4 + p_2^2 p_3^2) + 2d (p_1 p_3 p_4^2 + p_2 p_3^2 p_4) + e p_3^2 p_4^2, \\
    D(u,v) &= a p_1 p_2^3 + b (p_2^3 p_3 + 3 p_1 p_2^2 p_4) + 3c (p_1 p_2 p_4^2 + p_2^2 p_3 p_4) + d (p_1 p_4^3 + 3 p_2 p_3 p_4^2) + e p_3 p_4^3, \\
    E(u,v) &= a p_2^4 + 4b p_2^3 p_4 + 6c p_2^2 p_4^2 + 4d p_2 p_4^3 + e p_4^4.
    \end{aligned}
\end{equation*}
At the fixed point, $a,b,c,d,e$ are constants. Under the induced action on the principal symbol,
\[
f'_4(\xi,\eta)=f_4(P(\xi,\eta)^T)
=A\xi^4+4B\xi^3\eta+6C\xi^2\eta^2+4D\xi\eta^3+E\eta^4.
\]
Thus, the association between $L_4$ and $f_4$ preserves the $\GL_2(\R)$ action and establishes a one-to-one correspondence between their canonical forms.
\end{proof}

\begin{remark}\label{rmk:pde}
The equivalence established in Propositions~\ref{prop:pde_third} and~\ref{prop:pde_fourth} holds in both real and complex domains. However, eigenvalue-based classification is insufficient for real PDEs, which require distinguishing between real orbits as shown in the tables above.
\end{remark}

\section{Conclusion}\label{sec:conclusion}

This paper addresses Qi's fundamental questions regarding the spectral classification of symmetric tensors in the two-dimensional setting. We have investigated the equivalence classes of $2\times2\times2$ and $2\times2\times2\times2$ complex symmetric tensors through their spectral data, specifically the number of eigenpair classes and the multiplicities of zero eigenvalues. We have proved that these spectral data distinguish the canonical types under generic transformations in the two-dimensional complex setting, while they are insufficient to distinguish all equivalence classes, as illustrated by the parameterized fourth-order family. Over~$\R$, distinct equivalence classes can likewise share identical spectral counts due to the loss of information regarding real root structures. Furthermore, we have addressed Qi's second question by extending these results to derive canonical forms for real third- and fourth-order linear partial differential equations in two variables. Future research should aim to address Qi's remaining open problems, particularly the classification of third- and fourth-order tensors in three dimensions and higher. Unlike the two-dimensional case, higher-dimensional tensors possess significantly more complex geometric structures. Consequently, investigating whether spectral invariants can characterize these higher-dimensional orbits remains a challenging and vital direction for the field.

%\section*{Use of AI tools declaration}
%The authors declare they have not used Artificial Intelligence (AI) tools in the creation of this article.

%\section*{Acknowledgments}
%L. Fang was partially supported by the Natural Science Foundation of Xiamen (Grant No. 3502Z202371014).

%H.-L. Huang was partially supported by the Key Program of the Natural Science Foundation of Fujian Province (Grant No. 2024J02018) and the National Natural Science Foundation of China (Grant No. 12371037).

%\section*{Statements and Declarations}
\section*{Conflict of interest}
The authors declare no conflicts of interest.

\begin{spacing}{.88}
\setlength{\bibsep}{2.pt}
\bibliographystyle{abbrvnat}
\bibliography{tensor_classification}

\end{spacing}
\end{document}